\documentclass[12pt,reqno]{amsart}
\usepackage{url}
\usepackage{graphics}

\newtheorem{Theorem}{Theorem}[section]
\newtheorem{Lemma}[Theorem]{Lemma}
\newtheorem{Corollary}[Theorem]{Corollary}

\newtheorem{Question}[Theorem]{Question}

\def\qed{\ifhmode\textqed\fi
\ifmmode\ifinner\hfill\quad\qedsymbol\else\dispqed\fi\fi}
\def\textqed{\unskip\nobreak\penalty50
\hskip2em\hbox{}\nobreak\hfill\qedsymbol
\parfillskip=0pt \finalhyphendemerits=0}
\def\dispqed{\rlap{\qquad\qedsymbol}}

\def\supp{\textup{supp}}
\def\lcm{\textup{lcm}}
\def\lex{\textup{lex}}

\begin{document}
	
\title{A new proof\\ of the Herzog-Hibi-Zheng theorem}	
\author{Antonino Ficarra}

\address{Antonino Ficarra, Departamento de Matem\'{a}tica, Escola de Ci\^{e}ncias e Tecnologia, Centro de Investiga\c{c}\~{a}o, Matem\'{a}tica e Aplica\c{c}\~{o}es, Instituto de Investiga\c{c}\~{a}o e Forma\c{c}\~{a}o Avan\c{c}ada, Universidade de \'{E}vora, Rua Rom\~{a}o Ramalho, 59, P--7000--671 \'{E}vora, Portugal}
\email{antonino.ficarra@uevora.pt\,\,\,\,\,\,\,antficarra@unime.it}

\thanks{
}



\begin{abstract}
	We give a new, elementary proof of the celebrated Herzog-Hibi-Zheng theorem on powers of quadratic monomial ideals.
\end{abstract}

\maketitle

\vspace*{-0.5cm}
\section{Introduction}

Let $S=K[x_1,\dots,x_n]$ be the standard graded polynomial ring over a field $K$, and let $I\subset S$ be a monomial ideal. One of the most fascinating problems in Combinatorial Commutative Algebra is to decide when $I$ has \textit{linear powers}. That is, $I$ is generated in a single degree $d$, and all powers $I^k$ have a $dk$-linear resolution. Naively, one would expect that if $I$ has linear resolution, the same should be true for all its powers $I^k$. Very early on, counterexamples to this expectation were discovered by Terai, and, independently, by Sturmfels. See the introduction of \cite{HHZ2004}.

Nonetheless, using the classical Dirac theorem on chordal graphs \cite{Dirac61}, and Fr\"oberg theorem on edge ideals \cite{Froberg88}, the following influential result was proved in \cite{HHZ2004}.

\begin{Theorem}\label{Thm:HHZ}
	\textup{(Herzog-Hibi-Zheng).} Let $I\subset S$ be a quadratic monomial ideal. The following conditions are equivalent.
	\begin{enumerate}
		\item[\textup{(a)}] $I$ has linear resolution.
		\item[\textup{(b)}] $I$ has linear powers.
		\item[\textup{(c)}] All powers of $I$ have linear quotients.
	\end{enumerate}
\end{Theorem}

We will explain any unexplained concept in the next section.

The original proof given in \cite{HHZ2004}, which shows only that (a) and (b) are equivalent, made use of the so-called $x$-condition, which is a condition on the Gr\"obner basis of the defining ideal of the Rees algebra of $I$, and the delicate computation of the Graver basis of certain edge rings. Only in \cite[Theorem 3.1 and Corollary 3.2]{HH2011} it was realized that the conditions (a)-(b) are further equivalent to (c).

In the last 20 years, Theorem \ref{Thm:HHZ} inspired a vast amount of researches going into various directions. One of the most remarkable ones, is the conjecture of Nevo and Peeva \cite{NP}, which is still open, that aims to characterize those quadratic monomial ideals $I\subset S$ such that $I^k$ has linear resolution for all $k\gg0$.

Besides of an alternative, but still rather tedious proof of Theorem \ref{Thm:HHZ}, which is due to Banerjee \cite[Theorem 6.16]{AB}, no simple and elementary proof of Theorem \ref{Thm:HHZ} has been known. The goal of this note is to provide a short, simple and yet new proof of Theorem \ref{Thm:HHZ}. Moreover, a novel consequence of Theorem \ref{Thm:HHZ} will be derived in Corollary \ref{Cor:PI}, together with some related questions.

\section{From Dirac to Fr\"oberg to Herzog-Hibi-Zheng}

Let $G$ be a finite simple graph on the vertex set $\{x_1,\dots,x_n\}$, with edge set $E(G)$. The \textit{edge ideal} of $G$ is the monomial ideal $I(G)\subset S$ generated by those monomials $x_ix_j$ such that $\{x_i,x_j\}\in E(G)$. A graph $G$ is said \textit{complete} if every 2-subset $\{x_i,x_j\}$ of $V(G)$ is an edge of $G$. The \textit{open neighbourhood} of $x_i\in V(G)$ is the set
$$
N_G(x_i)\ =\ \{x_j\in V(G):\ \{x_i,x_j\}\in E(G)\}.
$$

A graph $G$ is called \textit{chordal} if it has no induced cycles of length bigger than three. Recall that a \textit{perfect elimination order} of $G$ is an ordering $x_1>\dots>x_n$ of its vertex set $V(G)$ such that $N_{G_i}(x_i)$ induces a complete subgraph on $G_i$, where $G_i$ is the induced subgraph of $G$ on the vertex set $\{x_i,x_{i+1},\dots,x_n\}$.

\begin{Theorem}\label{Thm:Dirac}
	\textup{(Dirac).} A finite 
	 simple graph $G$ is chordal, if and only if, $G$ admits a perfect elimination order.
\end{Theorem}

The \textit{complementary graph} $G^c$ of $G$ is the graph with vertex set $V(G^c)=V(G)$ and where $\{x_i,x_j\}$ is an edge of $G^c$ if and only if $\{x_i,x_j\}\notin E(G)$. A graph $G$ is called \textit{cochordal} if and only if $G^c$ is chordal.

\begin{Theorem}\label{Thm:Froberg}
	\textup{(Fr\"oberg).} Let $G$ be a simple finite graph. Then, $I(G)$ has a linear resolution, if and only if, $G$ is cochordal.
\end{Theorem}

Let $I\subset S$ be a monomial ideal. We denote by $\mathcal{G}(I)$ the (unique) minimal monomial generating set of $I$. We say that $I$ has \textit{linear quotients} if there exists an order $u_1>\dots>u_m$ of $\mathcal{G}(I)$, called a \textit{linear quotients order} of $I$, such that the colon ideals $(u_1,\dots,u_{i-1}):u_i$ are generated by variables, for all $1\le i\le m$.

Note that $(u_1,\dots,u_{i-1}):u_i$ is generated by the monomials $u_j:u_i=\lcm(u_j,u_i)/u_i$ where $1\le j<i$. Therefore, $I$ has linear quotients order $u_1>\dots>u_m$, if and only if, for all $j<i$ there exists $k<i$ such that $u_k:u_i=x_b$ is a variable dividing $u_j:u_i$.

It is well-known that if $I\subset S$ is an equigenerated monomial ideal with linear quotients, then $I$ has linear resolution \cite[Proposition 8.2.1]{HH2011}.

We recall the \textit{polarization} technique. For a monomial $u=x_1^{a_1}\cdots x_n^{a_n}\in S$, the \textit{polarization} of $u$ is the monomial $u^\wp=\prod_{i=1}^n(\prod_{j=1}^{a_i}x_{i,j})$ in the polynomial ring $K[x_{i,j}:1\le i\le n,1\le j\le a_i]$. The \textit{polarization} of a monomial ideal $I\subset S$ is defined to be the squarefree ideal $I^\wp$ with minimal generating set $\mathcal{G}(I^\wp)=\{u^\wp:u\in\mathcal{G}(I)\}$, in the polynomial ring $S^\wp$ over $K$ in the variables $x_{i,j}$ required to define $\mathcal{G}(I^\wp)$.

The following well-known property can be found in \cite[Corollary 1.6.3]{HH2011}.

\begin{Lemma}\label{Lem:pol}
	Let $I\subset S$ be a monomial ideal. Then $I$ has linear resolution, if and only if, $I^\wp$ has linear resolution.
\end{Lemma}

The following technical lemma will be crucial for the proof of Theorem \ref{Thm:HHZ}.
\begin{Lemma}\label{Lem:useful}
	Let $I\subset S$ be a quadratic monomial ideal having linear resolution. Then, up to relabeling, we can write $I=x_1P+J$ with $J\subset P$, where
	\begin{align}
		\label{eq:1}P\ &=\ (u/x_1:\ u\in\mathcal{G}(I)\,\,\text{such that}\,\,x_1\,\,\text{divides}\,\,u),\,\,\,\,\,\,\text{and}\\
		\label{eq:2}J\ &=\ (u\in\mathcal{G}(I):\ x_1\,\,\text{does not divide}\,\,u),
	\end{align}
	and moreover $J$ has again linear resolution.
\end{Lemma}
\begin{proof}
	By Lemma \ref{Lem:pol}, there exists a finite simple graph $G$ for which $I^\wp=I(G)$ is an edge ideal with linear resolution. Then, by Theorem \ref{Thm:Froberg}, $G$ is a cochordal graph. Fix $x_1>x_2>\dots>x_n$ a perfect elimination order of $G^c$. Let $P'=(x_j:x_j\in N_G(x_1))$ and $J'=I(G\setminus\{x_1\})$, where $G\setminus\{x_1\}$ is the graph obtained from $G$ by removing $x_1$ from $V(G)$ and all edges incident with $x_1$ from $E(G)$. Since $x_2>\dots>x_n$ is again a perfect elimination order of $(G\setminus\{x_1\})^c$, it follows by Theorem \ref{Thm:Dirac} that $G\setminus\{x_1\}$ is cochordal and by Theorem \ref{Thm:Froberg} that $J'$ has linear resolution. We have
	$$
	E(G)\ =\ \{\{x_1,x_j\}:\ x_j\in N_G(x_1)\}\cup E(G\setminus\{x_1\}).
	$$
	So $I(G)=x_1P'+J'$. We claim that $J'\subset P'$. Let $x_ix_j\in J'$ be a monomial generator corresponding to an edge $\{x_i,x_j\}\in E(G\setminus\{x_1\})\subset E(G)$. We must show that either $x_i$ or $x_j$ belongs to $N_G(x_1)$. If this was not the case, then $\{x_1,x_i\},\{x_1,x_j\}\in E(G^c)$. Since $x_1>x_2>\dots>x_n$ is a perfect elimination order of $G^c$, it would follow that $\{x_i,x_j\}\in E(G^c)$, against the fact that $\{x_i,x_j\}\in E(G)$. Hence $J'\subset P'$.
	
	Now, let $P$ and $J$ as defined in the equations (\ref{eq:1}) and (\ref{eq:2}). It is then clear that $I=x_1P+J$, $P^\wp=P'$ and $J^\wp=J'$. Hence $J\subset P$, because $J'\subset P'$. Finally, applying Lemma \ref{Lem:pol}, $J$ has a linear resolution, because $J'=J^\wp$ has linear resolution.
\end{proof}

For the proof of the next result we recall some concepts. Let $u=x_1^{a_1}\cdots x_n^{a_n}$ and $v=x_1^{b_1}\cdots x_n^{b_n}$. The \textit{lex order} $<_{\lex}$ is the monomial order of $S$ defined by setting $u>v$ if $a_j=b_j$ for all $j<i$ and $a_i>b_i$. The \textit{support} of a monomial ideal $I\subset S$ is defined as the set $\supp\,I=\bigcup_{u\in\mathcal{G}(I)}\supp(u)$, where $\supp(w)=\{x_i:\ x_i\,\,\text{divides}\,\,w\}$ for any monomial $w\in S$. 

\begin{Corollary}\label{Cor:useful}
	Let $I\subset S$ be a quadratic monomial ideal having linear resolution. Then, up to relabeling, the following two properties hold.
	\begin{enumerate}
		\item[$(*)$] If $i<j<k$ and $x_jx_k\in I$, then $x_ix_j\in I$ or $x_ix_k\in I$.\smallskip
		\item[$(**)$] If $x_i^2\in I$ and $x_jx_k\in I$ for some $j<i$  and some $k$, then $x_ix_j\in I$ or $x_ix_k\in I$.
	\end{enumerate}
    In particular, up to relabeling, $I$ has linear quotients with respect to the lex order of its minimal generators.
\end{Corollary}
\begin{proof}
	We proceed by induction on $|\supp\,I|$. If $|\supp\,I|=1$, there is nothing to prove. Let $|\supp\,I|>1$. By Lemma \ref{Lem:useful}, $I=x_1P+J$ where $P$ and $J$ are as in equations (\ref{eq:1}) and (\ref{eq:2}), $J\subset P$ and $J$ has linear resolution. Since $|\supp\,J|<|\supp\,I|$, by the inductive hypothesis $J$ satisfies the properties $(*)$ and $(**)$.
	
	\textit{Proof of} $(*)$. Let $i<j<k$ such that $x_jx_k\in I$. If $i>1$, then the property $(*)$ holds for $J$, and so it holds for $I$, too. Suppose now $i=1$. Then $x_jx_k\in J\subset P$. Thus $x_j\in P$ or $x_k\in P$, and so $x_ix_j\in I$ or $x_ix_k\in I$, as desired.
	
	\textit{Proof of} $(**)$. Let $i<j$ and $k$ such that $x_i^2\in I$ and $x_jx_k\in I$. Since $i>1$, we have $x_i^2\in J\subset P$, and so $x_i\in P$. If $j,k>1$, then the property $(**)$ holds for $J$, and so it holds for $I$ too. Otherwise, if $j=1$ or $k=1$, then $x_jx_i\in x_1P\subset I$ or $x_kx_i\in x_1P\subset I$, as desired.
	
	Finally, let $\mathcal{G}(I)=\{e_1,\dots,e_m\}$ be ordered such that $e_1>_{\lex}\cdots>_{\lex}e_m$. We claim that $e_1>\dots>e_m$ is a linear quotients order of $I$. If $|\supp\,I|=1$, there is nothing to prove. Let $|\supp\,I|>1$, and $I=x_1P+J$ as in the beginning of the proof. Then $\{e_1,\dots,e_s\}=\mathcal{G}(x_1P)$ and $\{e_{s+1},\dots,e_m\}=\mathcal{G}(J)$ for some $s$. It is clear that $(e_1,\dots,e_{r-1}):(e_r)$ is generated by variables for $r=2,\dots,s$. Now, let $r>s$. Since $x_1$ does not divide $e_{r}$ and $e_r\in P$, we obtain that
	$$
	(e_1,\dots,e_{r-1}):(e_r)=(x_1P,e_{s+1},\dots,e_{r-1}):(e_r)=(x_1)+(e_{s+1},\dots,e_{r-1}):(e_r).
	$$
	Since $|\supp\,J|<|\supp\,I|$, by induction $e_{s+1}>\dots>e_m$ is a linear quotients order of $J$. Hence, $(e_1,\dots,e_{r-1}):e_r$ is indeed generated by variables.
\end{proof}

\section{The new proof}

Let $I\subset S$ be a quadratic monomial ideal having linear resolution. After a suitable relabeling of the variables, Lemma \ref{Lem:useful} and Corollary \ref{Cor:useful} guarantee that $I$ satisfies the properties $(*)$ and $(**)$.

Let $<_\lex$ be the lex order on $S$ induced by $x_1>x_2>\dots>x_n$, and let $\mathcal{G}(I)=\{e_1<_\lex\dots<_\lex e_m\}$ ordered decreasingly according to $<_{\lex}$.

Let $y_1,\dots,y_m$ be a new set of variables, and let $\varphi:K[y_1,\dots,y_m]\rightarrow S$ be the $K$-linear map defined by setting $\varphi(y_i)=e_i$ for $i=1,\dots,m$. On the variables $y_1,\dots,y_m$ we consider the lex order induced by $y_1>\dots>y_m$.

Let $u\in\mathcal{G}(I^k)$ be a monomial of degree $2k$. Then $u=e_{i_1}\cdots e_{i_k}$ for certain integers $i_1\le\dots\le i_k$. Such a presentation is not unique in general. Following ideas given in \cite[Section 2]{HH2005}, we say that $u=e_{i_1}\cdots e_{i_k}$ is the \textit{standard presentation} of $u$ if $y_{i_1}\cdots y_{i_k}$ is the smallest monomial, with respect to $<$, such that $u=\varphi(y_{i_1}\cdots y_{i_k})$.

Fix $k\ge1$. On the set $\mathcal{G}(I^k)$ we consider the following order. Let $u,v\in\mathcal{G}(I^k)$, and let $u=e_{i_1}\cdots e_{i_k}$, $v=e_{j_1}\cdots e_{j_k}$ be the standard presentations of $u$ and $v$. We put $v>u$ if $y_{i_1}\cdots y_{i_k}<y_{j_1}\cdots y_{j_k}$ with respect to the lex order induced by $y_1>\dots>y_m$. We will show that $I^k$ has linear quotients with respect to this order for all $k\ge1$.

We are now ready to deliver the new elementary proof of Theorem \ref{Thm:HHZ}.

\begin{proof}[Proof of Theorem \ref{Thm:HHZ}]
	(c) $\Rightarrow$ (b) $\Rightarrow$ (a) is true for any equigenerated monomial ideal.\smallskip
	
	To prove (a) $\Rightarrow$ (c), we proceed by induction on $k\ge1$ and show that $I^k$ has linear quotients with respect to the order introduced above.
	
	Let $k=1$. Then, it follows from Corollary \ref{Cor:useful} that $I$ has linear quotients order $e_m>_{\lex}e_{m-1}>_{\lex}\cdots>_{\lex}e_1$. Since each monomial $e_i$ is its standard presentation and $y_m<y_{m-1}<\dots<y_1$, the base case of the induction is verified.
	
	Now, let $k>1$, and let $u=e_{i_1}\cdots e_{i_k}$ and $v=e_{j_1}\cdots e_{j_k}$ be the standard presentations of two monomials $u,v\in\mathcal{G}(I^k)$ with $v<u$. Our job is to find a monomial $w\in\mathcal{G}(I^k)$ with $w<u$ such that $w:v$ is a variable that divides $v:u$.
	
	If $\deg(v:u)=1$, there is nothing to prove. Suppose now $\deg(v:u)>1$.
	
	We may assume that $i_r\ne j_s$ for all $1\le s,r\le k$. Indeed, suppose $i_r=j_s$ for some $r$ and $s$. Set $u'=u/e_{i_r}$ and $v'=v/e_{j_s}$. Then $v':u'=v:u$. We claim that $u'=e_{i_1}\cdots e_{i_{r-1}}e_{i_{r+1}}\cdots e_{i_k}$ and $v'=e_{j_1}\cdots e_{j_{s-1}}e_{j_{s+1}}\cdots e_{j_k}$ are again standard presentations. Suppose this was not the case for $u'$. Then, if $u'=e_{\ell_1}\cdots e_{\ell_{k-1}}$ is the standard presentation of $u'$, we have $y_{\ell_1}\cdots y_{\ell_{k-1}}<y_{i_1}\cdots y_{i_{r-1}}y_{i_{r+1}}\cdots y_{i_k}$. Multiplying by $y_{i_r}$ we obtain $y_{\ell_1}\cdots y_{\ell_{k-1}}y_{i_r}<y_{i_1}\cdots y_{i_k}$ and $\varphi(y_{\ell_1}\cdots y_{\ell_{k-1}}y_{i_r})=u$, against the fact that $u=e_{i_1}\cdots e_{i_k}$ is the standard presentation of $u$. We can proceed similarly for $v'$. Hence $v'<u'$. By induction on $k$, there exists $w'\in\mathcal{G}(I^{k-1})$ such that $w':u'$ is a variable that divides $v':u'$. Set $w=w'e_{i_r}$ and let $w'=e_{t_1}\cdots e_{t_{k-1}}$ and $w=e_{p_1}\cdots e_{p_k}$ be the standard presentations of $w'$ of $w$. We have $w<u$, because $y_{p_1}\cdots y_{p_k}\le y_{t_{1}}\cdots y_{t_{k-1}}y_{i_r}<y_{i_1}\cdots y_{i_k}$. Moreover, $w:u=w':u'$ is a variable that divides $v:u=v':u'$, as desired.
	
	Let $e_{i_1}$ be the monomial appearing in $u$ for which $y_{i_1}$ is the biggest variable appearing in $y_{i_1}\cdots y_{i_k}$. By our assumption $i_1\ne j_s$ for all $s$. We can write
	$$
	fu=gv,\quad\quad\textup{where}\quad f=v:u\quad\textup{and}\quad g=u:v.
	$$
	
	Since $\deg(u)=\deg(v)=2k$, we have $\deg(f)=\deg(g)=d\ge2$. Let $x_{n+1}$ be a ``dummy" variable. Set $f'=fx_{n+1}^d$ and $g'=gx_{n+1}^d$. Then $f'u=g'v$. We claim that there is a sequence of $2(k+d)$ indices $v_{1},\dots,v_{2(k+d)}$ with $v_{2(k+d)+1}=v_1$ such that
	\begin{enumerate}
		\item[(i)] $x_{v_1}x_{v_2}=e_{i_1}$ with $v_1\le v_2$,
		\item[(ii)] $f'u=\prod_{\ell=1}^{k+d}(x_{v_{2\ell-1}}x_{v_{2\ell}})$ and $\,g'v=\prod_{\ell=1}^{k+d}(x_{v_{2\ell}}x_{v_{2\ell+1}})$,
		\item[(iii)] if both $v_{\ell},v_{\ell+1}\ne n+1$, then
		$$
		x_{v_{\ell}}x_{v_{\ell+1}}\in\begin{cases}
			\,\{e_{i_1},\dots,e_{i_k}\},&\textup{if}\ \ell\ \textup{is odd},\\
			\,\{e_{j_1},\dots,e_{j_k}\},&\textup{if}\ \ell\ \textup{is even}.
		\end{cases}
		$$
	\end{enumerate}
    \textit{Proof of the Claim.} The trick which we use and which is well-known to the experts, is to associate to the relation $f'u=g'v$ a so-called \textit{even closed walk} of a certain graph, see also \cite[Lemma 10.1.4]{HH2011}. Let $e_{i_1}=x_{v_1}x_{v_2}$ with $v_1\le v_2$. Since $x_{v_2}$ divides $f'u=g'v$, then $x_{v_2}$ divides $g'$ or $v$. If $x_{v_2}$ divides $g'$, we can write $g'v=(x_{v_2}x_{n+1})h$ where $h$ is a suitable monomial of degree $2(k+d-1)$. Since $x_{n+1}$ divides $f'$ we can find $x_{v_3}$ dividing $f$ and then $x_{v_3}$ divides $h$. Otherwise, $x_{v_2}$ divides $v$, so it divides $e_{j_s}$ for some $s$, and $e_{j_s}=x_{v_2}x_{v_3}$. Iterating this reasoning, it is clear that we can find the desired sequence of indices satisfying the properties (i)-(ii)-(iii). $\square$\smallskip

    Let $\ell_1<\ell_2<\cdots$ be the integers $3\le\ell\le 2(k+d)$ for which $v_\ell=n+1$. There are at least two such integers since $\deg(f)\ge2$. We distinguish the possible cases.\smallskip
    
    \textsc{Case 1.} Suppose $\ell_1$ is even. Let $u''=u/[(x_{v_1}x_{v_2})\cdots(x_{v_{\ell_1-3}}x_{v_{\ell_1-2}})]$ and set
    \begin{equation}\label{eq:pres}
    	w=(x_{v_2}x_{v_3})\cdots(x_{v_{\ell_1-2}}x_{v_{\ell_1-1}})u''.
    \end{equation}
    Then $w\in\mathcal{G}(I^k)$ and $w:u=x_{v_{\ell_1-1}}$. Since $e_{i_1}$ appears in the standard presentation of $u$, but not in that of $v$, by the property (iii) we have $e_{i_1}\ne (x_{v_2}x_{v_3}),\dots,(x_{v_{\ell_1-2}}x_{v_{\ell_1-1}})$. Let $e_{p_1}\cdots e_{p_k}$ be the standard presentation of $w$, and let $e_{t_1}\cdots e_{t_k}$ be the presentation given in (\ref{eq:pres}). Then,
    $y_{p_1}\cdots y_{p_k}\le y_{t_1}\cdots y_{t_k}<y_{i_1}\cdots y_{i_k}$ because the $y_{i_1}$-degree of $y_{i_1}\cdots y_{i_k}$ is strictly bigger than the $y_{i_1}$-degree of $y_{t_1}\cdots y_{t_k}$. We conclude that $w<u$. Since $x_{v_{\ell_1-1}}=w:u$ divides $f=v:u$, we are done in this case.\smallskip
    
    \textsc{Case 2.} Suppose both $\ell_1$ and $\ell_2$ are odd. Then, we consider the monomials $u'=u/[(x_{v_{\ell_1+2}}x_{v_{\ell_1+3}})\cdots(x_{v_{\ell_2-2}}x_{v_{\ell_2-1}})]$ and $v'=v/[(x_{v_{\ell_1+1}}x_{v_{\ell_1+2}})\cdots(x_{v_{\ell_2-3}}x_{v_{\ell_2-2}})]$. Notice that $\ell_2-\ell_1\ge4$. Otherwise, $\ell_2=\ell_1+2$ and then $x_{\ell_1+1}$ divides both $f$ and $g$, which is not possible since $\gcd(f,g)=\gcd(v:u,u:v)=1$. Therefore $\ell_2-\ell_1\ge4$, and so $\deg(u')=\deg(v')<k$. Since $e_{i_1}$ divides $u'$ again but not $v'$ and $u'$ is in standard presentation, it follows that $v'<u'$. Thus by inductive hypothesis, there exists $w'\in\mathcal{G}(I^s)$ where $s=\deg(u')/2$ with $w'<u'$ such that $w':u'$ is a variable that divides $v':u'=f/x_{v_{\ell_1+1}}$. Set $w=w'(x_{v_{\ell_1+2}}x_{v_{\ell_1+3}})\cdots(x_{v_{\ell_2-2}}x_{v_{\ell_2-1}})$. Then $w\in\mathcal{G}(I^k)$, $w<u$ and $w:u=w':u'$ is a variable that divides $f=v:u$.\smallskip
    
    \textsc{Case 3.} Suppose $\ell_1$ is odd and $\ell_2$ is even. Then, we consider the monomials $u'=u/[(x_{v_3}x_{v_4})\cdots(x_{v_{\ell_1-2}}x_{v_{\ell_1-1}})]$ and $v'=v/[(x_{v_2}x_{v_3})\cdots(x_{v_{\ell_1-3}}x_{v_{\ell_1-2}})]$. Notice that $v'<u'$. If $\deg(u')=\deg(v')<2k$, by induction there exists $w'\in\mathcal{G}(I^s)$ with $s=\deg(u')/2$ and $w'<u'$ such that $w':u'$ is variable that divides $v':u'=f$. Setting $w=w'(x_{v_3}x_{v_4})\cdots(x_{v_{\ell_1-2}}x_{v_{\ell_1-1}})$, we have $w\in\mathcal{G}(I^k)$ and $w<u$. Moreover, $w:u=w':u'$ is a variable that divides $f=v:u$, as desired.
    
    Suppose now that $u'=u$ and $v'=v$. Hence $\ell_1=3$ and $x_{v_4}$ divides $f$. We may furthermore assume that for none of the integers $4\le p\le\ell_2-1$ we have $v_p=v_1$ or $v_p=v_2$. Indeed, assume, for instance, that $v_p=v_1$ for some $4\le p\le\ell_2-1$. Suppose that $p$ is odd. Consider the monomial
    $$
    w=(x_{v_4}x_{v_5})\cdots (x_{v_{p-1}}x_{v_p})(x_{v_{p+1}}x_{v_1})[u/(x_{v_1}x_{v_2})(x_{v_5}x_{v_6})\cdots(x_{v_p}x_{v_{p+1}})]\in\mathcal{G}(I^k).
    $$
    We claim that $w<u$. Indeed, suppose that $v_{p+1}\ne v_2$. Then $e_{i_1}$ appears in the standard presentation of $u$ but not in the above presentation of $w$. This implies that $w<u$ in this case. Suppose now that $v_{p+1}=v_2$. Then $e_{i_1}^2$ appears in the standard presentation of $u$, but $e_{i_1}$ appears in degree one in the above presentation of $w$. This implies again that $w<u$. It is clear that $w:u=x_{v_4}$ divides $f=v:u$. We can proceed similarly if $v_p=v_1$ for $p$ even, or if $v_p=v_2$ for some $4\le p\le\ell_2-1$.
    
    Summarizing our argument thus far, we may assume that $\ell_1=3$ and that for all integers $4\le p\le\ell_2-1$ we have $v_p\ne v_1$ and $v_p\ne v_2$.\smallskip
    
    \textsc{Subcase 3.1.} Suppose there is an integer $4\le p<\ell_2-1$ such that $v_p\ne v_{p+1}$.
    
    \textsc{Subcase 3.1.1.} Let $v_1\ne v_2$. Assume that $x_{v_1}x_{v_{p}}\in I$ or $x_{v_1}x_{v_{p+1}}\in I$. For instance, say that $x_{v_1}x_{v_p}\in I$. Then, we consider the following monomial of $\mathcal{G}(I^k)$,
    \begin{equation}\label{eq:w}
    	w=(x_{v_4}x_{v_5})\cdots(x_{v_p}x_{v_1})\cdots(x_{v_{\ell_2-2}}x_{v_{\ell_2-1}})[u/(x_{v_1}x_{v_2})(x_{v_5}x_{v_6})\cdots(x_{v_{\ell_2-3}}x_{v_{\ell_2-2}})].
    \end{equation}
    Since $e_{i_1}$ appears in the standard presentation of $u$ but not in the above presentation of $w$, we have $w<u$. Moreover, $w:u=x_{v_4}$ divides $f$, as desired. We can proceed similarly if $x_{v_2}x_{v_{p+1}}\in I$ for some $4\le p<\ell_2-1$.
    
    Suppose now that $x_{v_1}x_{v_p},x_{v_1}x_{v_{p+1}}\notin I$. Since $x_{v_p}x_{v_{p+1}}\in I$, by the property $(*)$ it follows that either $v_p<v_1$ or $v_{p+1}<v_1$. Let $v_p<v_1$. Since $v_p<v_1<v_2$, $x_{v_1}x_{v_2}\in I$ and $x_{v_1}x_{v_p}\notin I$, again by the property $(*)$ we obtain that $x_{v_2}x_{v_p}\in I$. Then, we can consider the following monomial of $\mathcal{G}(I^k)$,
    $$
    w=(x_{v_4}x_{v_5})\cdots(x_{v_p}x_{v_2})\cdots(x_{v_{\ell_2-2}}x_{v_{\ell_2-1}})[u/(x_{v_1}x_{v_2})(x_{v_5}x_{v_6})\cdots(x_{v_{\ell_2-3}}x_{v_{\ell_2-2}})].
    $$
    As before, $w<u$ and $w:u=x_{p_4}$ divides $f=v:u$, as desired.\smallskip
    
    \textsc{Subcase 3.1.2.} Now, let $v_1=v_2$. Our assumption on $e_{i_1}$ ensures that $v_1\ge v_{p}$. Since $x_{v_p}x_{v_{p+1}}\in I$, the property $(**)$ implies that either $x_{v_1}x_{v_p}\in I$ or $x_{v_1}x_{v_{p+1}}\in I$. If $x_{v_1}x_{v_p}\in I$, we can consider again the monomial given in (\ref{eq:w}). We have $w<v$ and $w:v=x_{v_4}$ divides $f$. We can proceed similarly if $x_{v_1}x_{v_{p+1}}\in I$.\smallskip
    
    \textsc{Subcase 3.2.} Suppose now that $v_{p}=v_{p+1}$ for all $4\le p<\ell_2-1$. Since we assumed that $e_{i_r}\ne e_{j_s}$ for all $r$ and $s$, we conclude that $\ell_2=6$ and $v_4=v_5$. In each of the cases $v_4<v_1<v_2$, $v_1<v_4<v_2$ and $v_4<v_1=v_2$, by using either $(*)$ or $(**)$, we have either $x_{v_1}x_{v_4}\in I$ or $x_{v_2}x_{v_4}\in I$. To conclude the proof it is enough to consider the monomial $w=(x_{v_i}x_{v_4})[u/(x_{v_1}x_{v_2})]\in\mathcal{G}(I^k)$, where $i=1$ if $x_{v_1}x_{v_4}\in I$, and $i=2$ if $x_{v_2}x_{v_4}\in I$. We have $w<u$ and $w:u=x_{v_4}$ divides $f$.
\end{proof}

Basser \textit{et. al.} also found another proof of Theorem \ref{Thm:HHZ} \cite[Corollary 3.11]{BDMS}.\medskip

We conclude the paper with the following consequence and some questions.

\begin{Corollary}\label{Cor:PI}
	Let $I\subset S$ be a quadratic monomial ideal with linear resolution and let $P\subset S$ be a monomial prime ideal containing $I$. Then $P^kI^\ell$ has linear quotients for all $k,\ell\ge1$.
\end{Corollary}
\begin{proof}
	We introduce the ``dummy" variable $x_0$, and up to a suitable extension, we may assume that $S=K[x_0,x_1,\dots,x_n]$ and that $\supp\,P\cup\supp\,I\subseteq\{x_1,\dots,x_n\}$. Let $J=x_0P+I$. This is a rather trivial example of a Betti splitting \cite{FHT}. Indeed, $J=x_0P+I$ is an $x_0$-splitting (in the sense of \cite{FHT}) because $x_0P$ and $I$ have 2-linear resolution. Since $x_0P\cap I=x_0(P\cap I)=x_0I$ has a $3$-linear resolution, it follows from \cite[Proposition 1.8]{CF} that $J$ is a quadratic monomial ideal having a $2$-linear resolution. By the proof of Theorem \ref{Thm:HHZ}, $J^k$ has linear quotients with respect to the order $<$ described in the beginning of the section.
	
	We now show that $P^kI^\ell$ has linear quotients for all $k,\ell\ge1$. It is equivalent to show that $x_0^kP^kI^\ell$ has linear quotients. Let $u,v\in\mathcal{G}(x_0^kP^kI^\ell)\subset\mathcal{G}(J^{k+\ell})$ with $v<u$ with respect to the order $<$. Then, there exists $w\in\mathcal{G}(J^{k+\ell})$ with $w<u$ such that $w:u=x_q$ for some $q$ and $x_q$ divides $v:u$. Since the $x_0$-degree of $u$ and $v$ is $k$, it follows that $x_q\ne x_0$ and the $x_0$-degree of $w$ is less or equal to $k$. Since $\mathcal{G}(J^{k+\ell})$ is the disjoint union $\bigsqcup_{i=0}^{k+\ell}\mathcal{G}(x_0^iP^iJ^{k+\ell-i})$, it follows that $w=x_0^iw_0w_1$ with $i\le k$, $w_0\in\mathcal{G}(P^i)$ and $w_1\in\mathcal{G}(I^{k+\ell-i})$. If $i=k$, then $w\in\mathcal{G}(x_0^kP^kI^\ell)$, as desired. Suppose now that $i<k$. By assumption $I\subset P$. Write $w_1=e_1\cdots e_{k+\ell-i}$ with $e_j=x_{r_j}x_{s_j}\in I$ and $x_{s_j}\in P$ for all $j$. Then, we consider the monomial $w'=x_0^{k-i}w_0'w_1'$ with $w_0'=w_0x_{s_1}\cdots x_{s_{k-i}}\in\mathcal{G}(P^k)$ and $w_1'=w_1/(e_{1}\cdots e_{k-i})\in\mathcal{G}(I^\ell)$. It follows from the definition of $<$ that $w'<w$. Moreover, $w':u$ divides $w:u=x_q$. Hence $w':u=x_q$ too, and since $w'\in\mathcal{G}(x_0^kP^kI^\ell)$, this concludes the proof.
\end{proof}

This result is no longer valid if $I$ is not a quadratic monomial ideal with linear resolution, as the following example \cite[Example 4.3]{CH} of Conca and Herzog shows. Let $S=K[a,b,c,d]$, $I=(a^2b,abc,bcd,cd^2)$ and $P=(b,c)$. Then $I$ has linear quotients and $I\subset P$, but $PI$ does not have linear quotients, not even linear resolution.

On the other hand, very recently it was shown in \cite[Lemma 3.3]{FMR} that $PI$ has linear quotients for any edge ideal $I$ with linear resolution and any monomial prime ideal $P$, independently from the assumption that $I\subset P$. Hence, we are left to ask whether the assumption that $I\subset P$ is really needed in Corollary \ref{Cor:PI}. This raises the following question.
\begin{Question}
	Let $I\subset S$ be a quadratic monomial ideal with linear resolution, and let $P\subset S$ be a monomial prime ideal. Is it true that $P^kI^\ell$ has linear quotients for all $k,\ell\ge1$ ?
\end{Question}

\textbf{Acknowledgment.} The author was partly supported by the Grant JDC2023-051705-I funded by MICIU/AEI/10.13039/501100011033 and by the FSE+. Moreover, the author is grateful to Than Vu and Dang Hop Nguyen for some comments on an earlier draft of the manuscript and to Somayeh Moradi for her encouragement to complete this project.


\begin{thebibliography}{99}

\bibitem{AB} A. Banerjee, \textit{The regularity of powers of edge ideals}, J. Algebraic Combin. {\bf41}(2015), no. 2, pp. 303--321.

\bibitem{BDMS} E. Basser, R. Diethorn, R. Miranda, M. Stinson-Maas, \textit{Powers of Edge Ideals with Linear Quotients}, 2024, preprint \url{https://arxiv.org/abs/2412.03468}

\bibitem{CH} A. Conca, J. Herzog. \textit{Castelnuovo--Mumford regularity of products of ideals}. Collect. Math., {\bf 54}, 137--152 (2003).

\bibitem{CF} M. Crupi, A. Ficarra, \textit{Linear resolutions of t-spread lexsegment ideals via Betti splittings}, Journal of Algebra and Its Applications, doi 10.1142/S0219498824500725

\bibitem{Dirac61} G.A. Dirac, \textit{On rigid circuit graphs}, Abh. Math. Sem. Univ. Hamburg, {\bf 38} (1961), 71--76.

\bibitem{FMR} A. Ficarra, S. Moradi, T. R\"omer, \textit{Componentwise linear symbolic powers of edge ideals and Minh's conjecture}, (2024), preprint \url{https://arxiv.org/abs/2411.11537}.

\bibitem{FHT} C. A. Francisco, H. T. Ha, A. Van Tuyl, \textit{Splittings of monomial ideals}, Proc. Amer. Math. Soc., {\bf 137} (10) (2009), 3271-3282.

\bibitem{Froberg88} R. Fr\"oberg, \textit{On Stanley-Reisner rings}, Topics in algebra, Part 2 (Warsaw, 1988), 57--70, Banach Center Publ., 26, Part 2, PWN, Warsaw, 1990.

\bibitem{HH2011} J. Herzog, T. Hibi, \emph{Monomial ideals}, Graduate texts in Mathematics {\bf 260}, Springer, 2011.

\bibitem{HH2005} J. Herzog, T. Hibi, \textit{The depth of powers of an ideal}, J. Algebra 291 (2005), no. 2, 534--550.

\bibitem{HHZ2004} J. Herzog, T. Hibi and X. Zheng, \textit{Monomial ideals whose powers have a linear resolution},	Math. Scand. (2004), 23--32.

\bibitem{NP} E. Nevo, I. Peeva, \textit{$C_4$-free edge ideals}, J. Algebr. Combin. 37, 243–248 (2013).

\end{thebibliography}
\end{document}